\documentclass[leqno]{article}

\usepackage{amsmath, amsthm, amsfonts, amssymb, latexsym, amscd, enumitem, mathabx, stackrel,  url}

\usepackage{xy}
\xyoption{all}

\makeatletter
\def\url@smallstyle{%
  \@ifundefined{selectfont}{\def\UrlFont{\sf}}{\def\UrlFont{\small\ttfamily}}}
\makeatother
\numberwithin{equation}{section}

\theoremstyle{plain}
    \newtheorem{theorem}[equation]{Theorem}
    \newtheorem{lemma}[equation]{Lemma}
    \newtheorem{corollary}[equation]{Corollary}
    \newtheorem{proposition}[equation]{Proposition}

    \newtheorem*{theorem*}{Theorem}
    \newtheorem*{proposition*}{Proposition}
    \newtheorem*{corollary*}{Corollary}
    \newtheorem*{lemma*}{Lemma}
    \newtheorem*{conjecture*}{Conjecture}
    \newtheorem{definition-theorem}[equation]{Definition/Theorem}
    \newtheorem{definition-lemma}[equation]{Definition/Lemma}
\theoremstyle{definition}
    \newtheorem{definition}[equation]{Definition}
    
    \newtheorem{examples}[equation]{Examples}

    \newtheorem{remark}[equation]{Remark}

    \newcommand{\C}{\mathbb{C}}
    
    \newcommand{\Z}{\mathbb{Z}}

    \renewcommand{\H}{{\mathcal H}}
   	\renewcommand{\phi}{\varphi}
    \renewcommand{\epsilon}{\varepsilon}

 \newcommand{\into}{\hookrightarrow}

\newcommand{\dd}{d} 

\newcommand{\restrict}{\big{\vert}}
\newcommand{\Centre}{\mathfrak Z}
\newcommand{\conductor}{\mathfrak c}

\newcommand{\argument}{\hspace{2pt}\underbar{\phantom{g}}\hspace{2pt}}

\DeclareMathOperator{\Hom}{Hom}
\DeclareMathOperator{\End}{End}
\DeclareMathOperator{\Trace}{Trace}

\DeclareMathOperator{\Ad}{Ad}

\newcommand{\SL}{\operatorname{SL}}
\newcommand{\GL}{\operatorname{GL}}
\DeclareMathOperator{\id}{id}

\DeclareMathOperator{\vol}{vol}

\DeclareMathOperator{\infl}{infl}

\DeclareMathOperator{\image}{image}
\DeclareMathOperator{\Class}{Cl}
\DeclareMathOperator{\Mod}{Mod}

\DeclareMathOperator{\ch}{ch}

\DeclareMathOperator{\ind}{ind}

\DeclareMathOperator{\pind}{i}
\DeclareMathOperator{\pres}{r}
\DeclareMathOperator{\Pind}{I}
\DeclareMathOperator{\Pres}{R}
\DeclareMathOperator{\triv}{triv}
\newcommand{\HH}{\operatorname{HH}}

\newcommand{\h}{\operatorname{H}}

\newcommand{\ol}{\overline}

\begin{document}
\title{Parahoric induction and chamber homology for $\SL_2$}
\author{Tyrone Crisp}
\maketitle

\begin{abstract}
We consider the special linear group $G=\SL_2$ over a $p$-adic field, and its diagonal subgroup $M\cong \GL_1$. Parabolic induction of representations from $M$ to $G$ induces a map in equivariant homology, from the Bruhat-Tits building of $M$ to that of $G$. We compute this map at the level of chain complexes, and show that it is given by parahoric induction (as defined by J.-F. Dat).
\end{abstract} 

\section*{Introduction}

Consider the special linear group $G=\SL_2(F)$ over a $p$-adic field $F$. \emph{Parabolic induction} is the functor $\pind_M^G$ which takes (smooth, complex) representations of the diagonal subgroup $M\subset G$, pulls them back to the upper-triangular subgroup $P$ along the quotient map $P\to M$, and then induces up to $G$. This construction is remarkably efficient: it generically preserves irreducibility, and the coincidences between the resulting representations of $G$ are few and (mostly) easily understood.

Now let $\mathcal O$ be the ring of integers of $F$. The functor producing representations of $K=\SL_2(\mathcal O)$ from representations of its diagonal subgroup $L$ according to the above recipe has fewer desirable properties: for example, the representations thus produced are infinite-dimensional, and therefore far from irreducible. Dat has proposed a replacement for $\pind_M^G$ in this context, called \emph{parahoric induction} \cite{Dat_Finitude}. 

The representation theory of $K$ (and of other compact open subgroups of reductive $p$-adic groups) is of interest not just for its own sake, but also in relation to the representation theory of $G$: see \cite{Bushnell-Kutzko_types}, for example. An appealing feature of Dat's construction is its compatibility with parabolic induction: there is a commutative diagram 
\[ \xymatrix@C=50pt{
\Mod(L) \ar[r]^-{\substack{\text{parahoric}\\ \text{induction}}} \ar[d]_-{\substack{\text{compact}\\ \text{induction}}} & \Mod(K) \ar[d]^-{\substack{\text{compact}\\ \text{induction}}} \\
\Mod(M) \ar[r]_-{\substack{\text{parabolic} \\ \text{induction}}} & \Mod(G) }
\]
of functors between categories of smooth representations. (Dat proves this for a general minimal Levi subgroup of a reductive group \cite[(1.4)]{Dat_Finitude}.)  

The main result of this paper is a commutative diagram of a similar kind:
\begin{equation}\label{intro_diagram}\tag{$\Asterisk$}
\xymatrix@C=50pt{
C_*^M(X_M) \ar[r]^-{\substack{\text{parahoric}\\ \text{induction}}} \ar@{-->}[d]_-{\substack{\text{compact}\\ \text{induction}}} & C_*^G(X_G) \ar@{-->}[d]^-{\substack{\text{compact}\\ \text{induction}}} \\
C_*(\Mod_f (M) ) \ar[r]_-{\substack{\text{parabolic} \\ \text{induction}}} & C_*(\Mod_f(G))}
\end{equation}
Here $X_G$ and $X_M$ denote the Bruhat-Tits buildings of $G$ and $M$. $C_*^G(X_G)$ is a complex of simplicial chains on $X_G/G$, the coefficients over a simplex $s$ being the representation ring of the isotropy group of $s$. (This is the canonical chain complex computing \emph{chamber homology} for $G$; see \cite{Baum-Connes-Higson}). $C_*^M(X_M)$ is the corresponding complex for the action of $M$ on $X_M$, whose isotropy groups are all equal to $L$. The map $C_*^M(X_M)\to C_*^G(X_G)$ combines the inclusion $X_M\into X_G$ with parahoric induction from $L$ to the isotropy subgroups of $G$.

In the bottom row of \eqref{intro_diagram}, the subscripts $f$ indicate the subcategories of finitely generated representations. $C_*$ here denotes the Hochschild complexes associated to these categories by Keller (\cite{Keller}; cf. \cite{Mitchell} and \cite{McCarthy}), and the map $C_*(\Mod_f(M))\to C_*(\Mod_f(G))$ is the one induced by the functor $\pind_M^G$. 
The vertical arrows are given, in degree zero, by inducing representations from the isotropy groups of vertices up to $G$ and $M$ respectively. In higher degree, these maps are defined only at the level of homology.

The commutativity of the diagram in degree zero is essentially Dat's result, for which we do not offer a new proof. The point of \eqref{intro_diagram} is that Dat's definition extends in a natural way to a map between chamber-homology complexes, which is still compatible with parabolic induction in higher degree. Since the homology groups for $\SL_2$ vanish in degree $\geq 2$, our extension of Dat's theorem is so far a modest one; partial results for $\SL_n$, discussed at the end of the paper, point toward a more ambitious generalisation. 

This paper has three sections. Section \ref{inflation_section} reviews Dat's construction, and presents a few new results, in a general setting; note, though, that unlike Dat we work only over $\C$. Section \ref{SL2_section} contains explicit calculations in the case of $\SL_2(\mathcal O)$. Section \ref{chamber_section} contains the main result, Theorem \ref{parahoric_chamber_theorem}, on the commutativity of \eqref{intro_diagram}. 
The theorem also gives a realisation in chamber homology of the Jacquet restriction functor $\pres^G_M$. This part is comparatively easy: the restriction map $C_*^G(X_G)\to C_*^M(X_M)$ is just the naive analogue of $\pres^G_M$ for compact subgroups. 

\bigskip

This research was partially supported under NSF grant DMS-1101382, and by the Danish National Research Foundation through the Centre for Symmetry and Deformation (DNRF92). Some of the results have previously appeared in the author's Ph.D. thesis, written at The Pennsylvania State University under the direction of Nigel Higson. 

\subsubsection*{Basic definitions, notation and conventions:} All vector spaces are over $\C$. If $G$ is a locally compact, totally disconnected group, then $\H(G)$  denotes the Hecke algebra of $G$, and $\Mod(G)$ is the category of (smooth) representations of $G$. The terminology is explained, for example, in \cite{Renard}. For a closed subgroup $H$ of $G$, and a representation $\rho$ of $H$, $\ind_H^G\rho$ is the space of locally constant, compactly supported functions $f:G\to\rho$ satisfying $f(hg)=hf(g)$ for all $g\in G$ and $h\in H$; $G$ acts on $\ind_H^G\rho$ by right translation. When $G$ is compact, $R(G)$ denotes the complexified representation ring of $G$, and $\Class^\infty(G)$ the space of locally constant class functions on $G$; the latter two spaces are isomorphic via the map sending a representation $\pi$ to its character $\ch_\pi$. We write $e_G$ for the function on $G$ with constant value $1/\vol(G)$. We make frequent appeal to ``the Mackey formula'' for the composition of restriction and induction functors; the version of this formula proved by Kutzko in \cite{Kutzko_Mackey} covers all of the cases that arise here. 

\section{Inflation for Groups with an Iwahori Decomposition}\label{inflation_section}

\subsubsection*{Definition and basic properties}

\begin{definition}  An \emph{Iwahori decomposition} of a compact totally disconnected group $J$ is a triple $(U,L,\ol{U})$ of closed subgroups of $J$, such that
\begin{enumerate}[label={\rm(\arabic*)}, leftmargin=0cm,itemindent=2em,labelwidth=\itemindent,labelsep=0cm,align=left]
 \item $L$ normalises $U$ and $\ol{U}$, and
\item The product map $U\times L\times\ol{U}\to J$ is a homeomorphism.
\end{enumerate}
(Note that if (2) holds, then thanks to (1) the same is true for any ordering of the factors $U$, $L$, and $\ol{U}$.) 
\end{definition}

The motivation for this definition comes from reductive $p$-adic groups: every such group $G$ contains arbitrarily small compact open subgroups which admit Iwahori decompositions compatible with the Levi decompositions of the parabolic subgroups of $G$. See \cite[V.5.2]{Renard} for a precise statement; the primordial example is \cite[\S 2.2]{Iwahori-Matsumoto}. 

The standard theory of invariant measures on homogeneous spaces (as in, e.g., \cite{Weil_Integration}) shows that:

\begin{lemma}\label{Haar_lemma} Let $J=UL\ol{U}$ be a group with an Iwahori decomposition, and let $\dd u$, $\dd l$ and $\dd \ol{u}$ be Haar measures on $U$, $L$ and $\ol{U}$ respectively. The product measure $\dd u\, \dd l\, \dd \ol{u}$ is a Haar measure on $J$.\hfill\qed
\end{lemma}

Let $J=UL\ol{U}$ be a group with a fixed Iwahori decomposition. From now on we assume that the Haar measures on $J$, $U$, $L$, and $\ol{U}$ are all normalised to have total volume $1$. The Hecke algebra $\H(J)$ is both a left and a right module over $\H(U)$, $\H(L)$, and $\H(\ol{U})$. Since $L$ normalises $U$ and $\ol{U}$, the action of $\H(L)$ commutes with the idempotents $e_U$ and $e_{\ol{U}}$.

The following definition is Dat's \cite{Dat_Finitude}.

\begin{definition} Consider the tensor-product functors
\[ \pind_{U,\ol{U}}:\Mod(L)\to\Mod(J)\qquad \pind_{U,\ol{U}}\rho= \H(J)e_{\ol{U}}e_U\otimes_{\H(L)}\rho\]
\[ \pres_{U,\ol{U}}:\Mod(J)\to\Mod(L)\qquad \pres_{U,\ol{U}}\pi= e_{U}e_{\ol{U}}\H(J)\otimes_{\H(J)}\pi.\]
\end{definition}

The following concrete realisations of $\pind_{U,\ol{U}}$ and $\pres_{U,\ol{U}}$ are sometimes useful. Let $\pind_U,\ \pind_{\ol{U}}:\Mod(L)\to\Mod(J)$ be the composite functors $\pind_U= \ind_{LU}^J \infl_L^{LU}$ and $\pind_{\ol{U}}= \ind_{L{\ol{U}}}^J \infl_{L}^{L\ol{U}}$, where for example $\infl_L^{LU}$ is the functor of inflation, i.e., pull-back along the quotient map $LU\to L$. Then $\pind_U\rho\cong \H(J)e_U\otimes_{\H(L)}\rho$, and likewise for $\pind_{\ol{U}}$ and $\H(J)e_{\ol{U}}$. Computing the map $\H(J)e_{\ol{U}}\xrightarrow{f\mapsto fe_U}\H(J)e_U$ in this picture, one finds that
\[ \pind_{U,\ol{U}}\rho\cong \image\left(\pind_{\ol{U}}\rho\xrightarrow{I_U} \pind_U\rho\right), \quad \text{where}\quad I_U(f)(j)=\int_U f(uj)\,\dd u.\]
Similarly, let $\pres_U,\ \pres_{\ol{U}}:\Mod(J)\to\Mod(L)$ be the functors
$\pres_U \pi=\pi^U$ (the $U$-invariants in $\pi$), and $\pres_{\ol{U}}\pi=\pi^{\ol{U}}$. Then
\[ \pres_{U,\ol{U}}\pi\cong \image\left( \pres_{\ol{U}}\pi\xrightarrow{e_U}\pres_{U}\pi\right),\quad\text{where}\quad e_U(x)=\int_U ux\, \dd u.\]

We shall use another characterisation of $\pind_{U,\ol{U}}$ and $\pres_{U,\ol{U}}$, based on the following observation. 

\begin{lemma}\label{Centre_lemma} The map $\Phi:\Centre(L)\to\Hom_{J\times L}(\H(J)e_{\ol{U}},\H(J)e_{U})$ defined by $\Phi(z):f\mapsto fe_{U}z$
is a $\Centre(L)$-linear isomorphism. 
\end{lemma}
\noindent (Here $\Centre(L)\cong \End_{L\times L}(\H(L))$ denotes the Bernstein centre of $L$ \cite{Bernstein-Deligne}.)

\begin{proof} Frobenius reciprocity gives $\Hom_J(\pind_{\ol{U}},\pind_{U})\cong \Hom_L(\pres_{U}\pind_{\ol{U}},\id)$. Evaluation of functions at the identity in $J$ gives a natural transformation $\pres_{U}\pind_{\ol{U}}\to \id_L$, which is an isomorphism because $J=\ol{U}LU$. Applying the Yoneda lemma, we obtain an isomorphism 
\[ \Centre(L)\cong \End_{L\times L}(\H(L))\xrightarrow[\cong]{\Phi}\Hom_{J\times L}(\H(J)e_{\ol{U}},\H(J)e_{U}),\] which is $\Centre(L)$-linear by the naturality of the construction. Computing the Frobenius reciprocity isomorphism explicitly, one finds that $\Phi(1):f\mapsto fe_{U}$.
\end{proof}

\begin{lemma}\label{i_Schur_lemma} If $\rho$ is an irreducible representation of $L$, then $\pind_{U,\ol{U}}\rho$ is the unique irreducible representation of $J$ common to both $\pind_U\rho$ and $\pind_{\ol{U}}\rho$. Moreover, $\pind_{U,\ol{U}}\rho$ has multiplicity one in both $\pind_U\rho$ and $\pind_{\ol{U}}\rho$.
\end{lemma}

\begin{proof} Lemma \ref{Centre_lemma} implies that $\Hom_J(\pind_{\ol{U}}\rho,\pind_U\rho)$ is one-dimensional, spanned by $I_U$. The result now follows from Schur's lemma.
\end{proof}

\begin{examples}
\begin{enumerate}[label={\rm(\arabic*)}, leftmargin=0cm,itemindent=2em,labelwidth=\itemindent,labelsep=0cm,align=left]
\item When $\ol{U}=\{1\}$ is trivial, so that $J\cong L\ltimes U$, one has $\pind_{U,\ol{U}}\cong \pind_U = \infl_L^J$, the usual inflation functor. 
\item Let $\triv_L$ be the trivial representation of $L$. Then $\triv_J$ sits inside both $\pind_U\triv_L$ and $\pind_{\ol{U}}\triv_L$, as the space of constant functions in each case. So $\pind_{U,\ol{U}}\triv_L=\triv_J$.
\item Let $\tilde{\rho}$ denote the (smooth) contragredient of $\rho$. Then $\widetilde{\pind_U\rho}\cong \pind_U\tilde{\rho}$, and likewise for $\pind_{\ol{U}}$, and so Lemma \ref{i_Schur_lemma} implies that $\widetilde{\pind_{U,\ol{U}}\rho}\cong \pind_{U,\ol{U}}\tilde{\rho}$.
\end{enumerate}
\end{examples}

\begin{definition-lemma}\label{z_definition} Let $\tilde{z}_{U,\ol{U}}\in\Centre(L)$ be the preimage of the map
\[ \H(J)e_{\ol{U}}\xrightarrow{f\mapsto fe_{U}e_{\ol{U}}e_{U}}\H(J)e_{U}\] under the isomorphism $\Phi$ of Lemma \ref{Centre_lemma}. Let $z_{U,\ol{U}}$ be the image of $\tilde{z}_{U,\ol{U}}$ under the involution $l\mapsto l^{-1}$ on $\Centre(L)$. Then $z_{U,\ol{U}}$ is invertible.
\end{definition-lemma}

\begin{proof} We must show that $\tilde{z}_{U,\ol{U}}$ acts as a nonzero scalar on each irreducible representation $\rho$ of $L$. Lemma \ref{i_Schur_lemma} ensures that \[ \pind_{\ol{U}}\rho\xrightarrow{I_U}\pind_U\rho\xrightarrow{I_{\ol{U}}}\pind_{\ol{U}}\rho \xrightarrow{I_U} \pind_U\rho\] restricts to a composition of nonzero intertwining operators of the irreducible representation $\pind_{U,\ol{U}}\rho$, and so this composition is nonzero by Schur's lemma.
\end{proof}

Other descriptions of $z$ appear in Proposition \ref{character_proposition} and Remark \ref{dihedral_remark}. Explicit formulae for the Iwahori subgroup in $\SL_2(F)$ are given in \cite[Section 2.4]{Dat_Finitude} and in Proposition \ref{SL2_iz_proposition}.

\begin{proposition}\label{z_proposition}{\rm \cite[Proposition 2.2]{Dat_Finitude}} The operator $z_{U,\ol{U}}^{-1}e_U e_{\ol{U}}$ acts as an idempotent on each representation of $J$.
\end{proposition}

\begin{proof} The definition of $\tilde{z}_{U,\ol{U}}$ ensures that
\[ f(e_{\ol{U}}e_U)^2=f\tilde{z}_{U,\ol{U}}e_{\ol{U}}e_U\]
for every $f\in \H(J)$. Applying the involution $j\mapsto j^{-1}$ on $\H(J)$ gives the desired result.
\end{proof}

The following basic properties of $\pind_{U,\ol{U}}$ and $\pres_{U,\ol{U}}$ follow easily from Proposition \ref{z_proposition}, as in \cite[Lemme 2.8 and Corollaire 2.9]{Dat_Finitude}:

\begin{proposition}\label{parahoric_basic_proposition}
\begin{enumerate}[label={\rm(\arabic*)}, leftmargin=0cm,itemindent=2em,labelwidth=\itemindent,labelsep=0cm,align=left]
\item There are isomorphisms $\pind_{U,\ol{U}}\cong \pind_{\ol{U},U}$ and $\pres_{U,\ol{U}}\cong \pres_{\ol{U},U}$.
\item $\pind_{U,\ol{U}}$ and $\pres_{U,\ol{U}}$ are mutual two-sided adjoints.
\item $\pres_{U,\ol{U}}\pind_{U,\ol{U}}\cong \id_L$.\hfill\qed
\end{enumerate}
\end{proposition}

We also obtain a counterpart to Lemma \ref{i_Schur_lemma} for $\pres_{U,\ol{U}}$:

\begin{lemma} \label{r_Schur_lemma}
Let $\pi$ be an irreducible representation of $J$. Then 
\[\dim \Hom_L(\pres_{\ol{U}}\pi,\pres_U\pi)=0 \text{ or } 1.\] If the dimension is zero, then $\pres_{U,\ol{U}}\pi=0$. If the dimension is $1$, then $\pres_{U,\ol{U}}\pi\cong \pres_U\pi\cong \pres_{\ol{U}}\pi$.
\end{lemma}

\begin{proof} First note that if $\Hom_L(\pres_{\ol{U}}\pi,\pres_U\pi)=0$, then in particular the map $e_U:\pres_{\ol{U}}\pi\to \pres_U\pi$ is zero, and so its image $\pres_{U,\ol{U}}\pi$ is zero. 

Now suppose that the intertwining space is nonzero. There exists an irreducible representation $\rho$ of $L$ common to both $\pres_U\pi$ and $\pres_{\ol{U}}\pi$, which by Frobenius reciprocity implies that $\pi$ is a common irreducible component of $\pind_U\rho$ and $\pind_{\ol{U}}\rho$. So by Lemma \ref{i_Schur_lemma}, $\pi\cong \pind_{U,\ol{U}}\rho$. Thus $\pres_U\pi$ is a nonzero quotient of the irreducible representation $\pres_U\pind_{\ol{U}}\rho\cong \rho$, so $\pres_U\pi\cong\rho$. Similarly, $\pres_{\ol{U}}\pi\cong \rho$, and so $\Hom_L(\pres_{\ol{U}}\pi,\pres_U\pi) \cong \End_L(\rho)$ is one-dimensional.
\end{proof}

\subsubsection*{Character formulae}

Let $J=UL\ol{U}$ be a compact totally disconnected group with an Iwahori decomposition. All the groups in question will be fixed throughout this section, and we write $\pind=\pind_{U,\ol{U}}$ and $\pres=\pres_{U,\ol{U}}$. Passing from representations $\pi$ to their characters $\ch_\pi$, we may view $\pind$ and $\pres$ as maps between the spaces $\Class^\infty(J)$ and $\Class^\infty(L)$ of class functions on $J$ and $L$.

For example, suppose that $J=UL$ (i.e., $\ol{U}=\{1\}$), so that $\pind$ is the the usual inflation of representations, while $\pres$ is the functor $\pi\mapsto\pi^U$. The action on characters is easily computed:
$\pind$ is given by pulling functions back along the quotient map $J\to L$, while $\pres$ is given by integration along the fibres of this map.

Returning to the general case, consider the map $\lambda=\lambda_{U,\ol{U}}:J\to L$ defined by $\lambda(ul\ol{u})=l$. Then define $\lambda_*:\Class^\infty(J)\to \Class^\infty(L)$ and $\lambda^*:\Class^\infty(L)\to \Class^\infty(J)$ by
\[ (\lambda_*\phi) (l)= \int_{U}\int_{\ol{U}} \phi(ul\ol{u})\, \dd \ol{u}\, \dd u \qquad \text{and} \qquad (\lambda^*\psi) (j)= \int_J\psi\big( \lambda(k^{-1}jk)\big) \ \dd k,\]
for $\phi\in \Class^\infty(J)$ and $\psi\in \Class^\infty(L)$.

\begin{proposition}\label{character_proposition}
Let $J=UL\ol{U}$ be a group with Iwahori decomposition, and let $z=z_{U,\ol{U}}\in\Centre(L)$ be as in Proposition \ref{z_proposition}.
\begin{enumerate}[label={\rm(\arabic*)}, leftmargin=0cm,itemindent=2em,labelwidth=\itemindent,labelsep=0cm,align=left]
\item The maps $\displaystyle \xymatrix@1{ \Class^\infty(J) \ar@<0.5ex>[r]^{\pres} & \Class^\infty(L) \ar@<0.5ex>[l]^{\pind}}$ are given by
$\pres = z^{-1}\lambda_*$ and $\pind=\lambda^*z^{-1}$.
\item For each irreducible representation $\rho$ of $L$, one has
$\displaystyle z(\rho)=\frac{\dim\rho}{\dim(\pind\rho)}$.
\end{enumerate}
\end{proposition} 

\begin{proof} We first consider $\pres$. For each irreducible $\pi$ of $J$, and each $l\in L$, 
\[ \lambda_*(\ch_\pi)(l)= \int_U\int_{\ol{U}} \Trace( \pi(lu\ol{u}))\, \dd \ol{u}\, \dd u = \Trace(\pi(l)\pi(e_U)\pi( e_{\ol{U}})).\]

If $\pres\pi=0$, then $\pi(e_U)\pi(e_{\ol{U}})=0$, and so $\pres(\ch_\pi)=z^{-1}\lambda_*(\pi)=0$.
On the other hand, suppose that $\pres\pi=\rho$. Then
\[\lambda_*(\ch_\pi)(l)=\Trace(\pi(l)\pi(e_U)\pi(e_{\ol U}) ) = z(\rho) \Trace\left( \pi(l) z(\rho)^{-1} \pi(e_U)\pi(e_{\ol{U}})\right),\] and $z(\rho)^{-1}\pi(e_U)\pi(e_{\ol{U}})$ is a projection of $\pi$ onto $\pres\pi$ (Proposition \ref{z_proposition}). So
\[ \lambda_*(\ch_\pi)(l)= z(\rho)\Trace\left(\pi(l)\restrict_{\pres\pi}\right)=z(\pres\pi)\ch_{\pres\pi}(l),\]
giving $\pres=z^{-1}\lambda_*$.

Now turn to the map $\pind$. We consider the usual inner products on $\Class^\infty(L)$ and $\Class^\infty(J)$: 
\[ \langle \psi_1,\psi_2\rangle_L = \int_L \psi_1(l)\ol{\psi_2(l)}\ \dd l\] for $\psi_1,\psi_2\in\Class^\infty(L)$, and similarly for $J$. The characters of irreducible representations constitute orthonormal bases for $\Class^\infty(J)$ and $\Class^\infty(L)$.

A straightforward computation with Lemma \ref{Haar_lemma} shows that for each $\psi\in\Class^\infty(L)$ and $\phi\in \Class^\infty(J)$, one has $\langle \lambda^*\psi,\phi\rangle_J=\langle \psi,\lambda_*\phi\rangle_L$.
Also, $\langle \pind\psi,\phi\rangle_J= \langle \psi, \pres\phi\rangle_L$, because the functors $\pind$ and $\pres$ are adjoints. Thus the formula $\pres=z^{-1}\lambda_*$ gives, upon taking adjoints, $\pind=\lambda^*\ol{z^{-1}}$, where $\ol{z^{-1}}$ denotes the complex conjugate of $z^{-1}$.
Noting that $(\lambda^*\psi)(1)=\psi(1)$ for all $\psi\in\Class^\infty(L)$, we find 
\[\dim(\pind\rho)=\pind(\ch_\rho)(1)=\lambda^*\ol{z^{-1}}(\ch_\rho)(1)=\ol{z^{-1}(\rho)}\lambda^*(\ch_\rho)(1)=\ol{z^{-1}}\dim\rho.\]
Therefore $\ol{z(\rho)}=\dim\rho/\dim(\pind\rho)$, which is real, and (2) follows. Putting $\ol{z}=z$ into $\pind=\lambda^*\ol{z^{-1}}$ completes the proof of (1).
\end{proof}

\begin{remark}\label{dihedral_remark} The number $z(\rho)$ may be interpreted as measuring the relative position of the idempotents $e_U$ and $e_{\ol{U}}$ in the representation $\pind \rho$, as we shall now explain.

Let $\pi$ be an irreducible representation of $J$, and choose a $J$-invariant inner product on $\pi$. The self-adjoint idempotents $P=\pi(e_U)$ and $Q=\pi(e_{\ol{U}})$ determine a finite-dimensional unitary representation of the infinite dihedral group $\Gamma=(\Z/2\Z)*(\Z/2\Z)$: the generating involutions $s_1,s_2\in\Gamma$ map to the self-adjoint unitary operators $2P-1$ and $2Q-1$, respectively. This representation $\pi\restrict_{\Gamma}$  of $\Gamma$ decomposes into a direct sum of isotypical components, and each isotypical component is stable under the action of $L$.
 
Recall the list of irreducible unitary representations of $\Gamma$: for each angle $\alpha\in [0,\pi/2]$ one forms the two-dimensional representation $\tau_\alpha$ in which $P$ and $Q$ are represented by the matrices
\[\tau_{\alpha}(P)=\begin{bmatrix}1&0\\0&0\end{bmatrix},\qquad\tau_{\alpha}(Q) =\begin{bmatrix}\cos^2\alpha&\cos\alpha\sin\alpha\\ \cos\alpha\sin\alpha&\sin^2\alpha\end{bmatrix}.\]
For $\alpha\in (0,\pi/2)$, the $\tau_\alpha$ are irreducible and mutually inequivalent. The representations $\tau_0$ and $\tau_{\frac{\pi}{2}}$ each decompose into one-dimensional summands:
$\tau_0 = \tau_0' \oplus \tau_0''$ and $\tau_{\frac{\pi}{2}}= \tau_{\frac{\pi}{2}}' \oplus  \tau_{\frac{\pi}{2}}''$. These four one-dimensional representations, together with the irreducible $\tau_\alpha$, form a complete list of the irreducible unitary representations of $\Gamma$.
(The list is obtained by expressing $\Gamma$ as a semidirect product $(\Z/2\Z)\ltimes \Z$, and applying  Mackey theory \cite[Section 14]{Mackey_InducedI}.)

Now, $\pres\pi$ is the range of $PQ$, and $PQ$ is nonzero only in $\tau_0'$ and in the $\tau_\alpha$ components for $\alpha\in(0,\pi/2)$. So 
\[ \pres\pi\neq 0 \iff \pi\restrict_{\Gamma}\text{ contains }\tau_0'\text{ or }\tau_\alpha\text{ for some }\alpha\in(0,\pi/2).\] Suppose $\pres\pi\neq 0$, so that $\pi=\pind\rho$ for some irreducible $\rho$ of $L$. Since $\pres\pi$ is an irreducible representation of $L$, and $L$ preserves the isotypical decomposition of $\pi\restrict_{\Gamma}$, it follows that $\pi\restrict_{\Gamma}$ contains exactly one of the representations $\tau_0'$ or $\tau_\alpha$ (possibly with multiplicity $>1$). We then have $PQP=\cos^2(\alpha)P$ (setting $\alpha=0$ if $\pi\restrict_{\Gamma}$ contains $\tau_0'$), which by the definition of $z$ implies that
\[ z(\rho) = \cos^2(\alpha).\]

Thus the formula $z(\rho)=\dim\rho/\dim(\pind\rho)$ imposes a restriction on the irreducible representations of $\Gamma$ that may occur in irreducible representations of $J$. For example, if $L$ is commutative, so that $\dim\rho=1$ for every irreducible $\rho$, then the representation $\tau_\alpha$ of $\Gamma$ may occur only in those irreducibles $\pi$ of $J$ having $\dim\pi = 1/\cos^2(\alpha)$.
\end{remark}

\section{The Iwahori Subgroup of $\SL_2(F)$}\label{SL2_section}

Let $F$ be a $p$-adic field, with ring of integers $\mathcal O$ and maximal ideal $\mathfrak p$. Choose a generator $\varpi$ for $\mathfrak p$. We write $\mathfrak f$ for the residue field $\mathcal O/\mathfrak p$, and $q$ for the cardinality of $\mathfrak f$.

Let $G=\SL_2(F)$, and consider the standard Iwahori subgroup \cite[\S 2.2]{Iwahori-Matsumoto}
\[ J=\begin{bmatrix} \mathcal O & \mathcal O\\ \mathfrak p & \mathcal O\end{bmatrix}.\] The notation means that $J$ is the group of determinant-one matrices whose bottom-left entry lies in $\mathfrak p$, and whose other entries lie in $\mathcal O$. (Similar notation will be used throughout the paper.) $J$ admits an Iwahori decomposition $J=UL\ol{U}$, where
\[ U=\begin{bmatrix} 1&\mathcal O\\ 0& 1\end{bmatrix},\quad L=\begin{bmatrix} \mathcal O^\times & 0 \\ 0 &\mathcal O^\times\end{bmatrix},\quad \ol{U}=\begin{bmatrix} 1& 0 \\ \mathfrak p & 1\end{bmatrix}.\] Throughout this section we write $\pind$, $\pres$ and $z$ for $\pind_{U,\ol{U}}$, $\pres_{U,\ol{U}}$ and $z_{U,\ol{U}}$. 

\subsubsection*{Computations of $\pind$, $\pres$, and $z$}

Let $\rho$ be an irreducible representation of $L$; identifying $L$ with $\mathcal O^\times$ via $\left[\begin{smallmatrix} a &\\ &a^{-1}\end{smallmatrix}\right]\mapsto a$, we view $\rho$ as a smooth homomorphism $\mathcal O^\times\to\C^\times$. If $\rho$ is trivial, then $\pind  \rho$ is the trivial representation of $J$. Assume that $\rho$ is nontrivial, and let $\conductor$ denote the \emph{conductor} of $\rho$:
\[ \conductor =\min\{n\geq 1\ |\ \rho\text{ is trivial on } 1+\mathfrak p^n\}.\]
Then define $J_{\conductor}=\left[\begin{smallmatrix}\mathcal O&\mathcal O\\ \mathfrak p^{\conductor} &\mathcal O\end{smallmatrix}\right]$, and let $\rho:J_{\conductor}\to\C^\times$ be the homomorphism $\rho\left[\begin{smallmatrix} a&b\\c&d\end{smallmatrix}\right]\coloneq \rho(a)$.

\begin{proposition}\label{SL2_iz_proposition}
\begin{enumerate}[label={\rm(\arabic*)}, leftmargin=0cm,itemindent=2em,labelwidth=\itemindent,labelsep=0cm,align=left]
\item $\pind \rho\cong\ind_{J_{\conductor}}^J\rho$.
\item $\displaystyle z(\rho) = \begin{cases} 1&\text{if $\rho$ is trivial,}\\ q^{1-\conductor}&\text{if $\rho$ is nontrivial with conductor $\conductor$.}\end{cases}$
\end{enumerate}
\end{proposition}

\begin{proof} A short computation shows that the image of $I_U:\pind_{\ol{U}}\rho\to\pind_{U} \rho$ lies in the subspace $\ind_{J_{\conductor}}^J\rho$, and so $\pind \rho\subseteq\ind_{J_{\conductor}}^J\rho$. Using the Mackey formula, and the minimality of $\conductor$, one can show that $\ind_{J_{\mathfrak c}}^J\rho$ is irreducible: see \cite[Lemma 9.2]{Baum-Higson-Plymen_SL2}. This proves part (1).

For part (2), Proposition \ref{character_proposition} gives $z(\rho)=\dim(\pind\rho)^{-1}$. For nontrivial $\rho$, part (1) implies that
\[ \dim(\pind\rho)=[J:J_{\conductor}]=[\mathfrak p:\mathfrak p^{\conductor}]=q^{\conductor-1}.\qedhere\]
\end{proof}

We now turn to the functor $\pres$. Let $t=\left[\begin{smallmatrix} \varpi^{-1} & 0\\ 0& \varpi\end{smallmatrix}\right]\in G$.  
If $\pi$ is a representation of a subgroup $H$ of a group $G$, and if $g\in G$, then $\pi^g$ denotes the representation $x\mapsto \pi(gxg^{-1})$ of the group $H^g\coloneq g^{-1}Hg$. 

\begin{lemma}\label{r_asymptotic_lemma} Let $\pi$ be a smooth, finite-dimensional representation of $J$. Then
\[ \Hom_L(\pres_{U} \pi,\pres_{\ol{U}}\pi)\cong \Hom_{J\cap J^{t^n}}(\pi,\pi^{t^n})\]
for all sufficiently large $n$. If $\pi$ is irreducible, then  
\[\pres\pi=0 \iff \Hom_{J\cap J^{t^n}}(\pi,\pi^{t^n})=0\quad\text{for all } n>>0.\] 
\end{lemma}

\begin{proof} To compactify the notation, let $J^n=J\cap J^{t^n}$. Explicitly, $J^n=\left[\begin{smallmatrix} \mathcal O & \mathfrak p^{2n}\\ \mathfrak p & \mathcal O\end{smallmatrix}\right]$. This group has an Iwahori decomposition $J^n=U^n L\ol{U}$, where $U^n\coloneq U^{t^n}$. 

Since $t^n$ centralises $L$, we have an isomorphism $\pres_U\pi=\pi^U\cong (\pi^{t^n})^{U^n}$ of representations of $L$, and so
\[ \Hom_L(\pi^{\ol{U}},\pi^U)\cong \Hom_L(\pi^{\ol{U}},(\pi^{t^n})^{U^n}).\]
Because $\pi$ is smooth and finite-dimensional, the kernel of $\pi$ contains some congruence subgroup $\left[\begin{smallmatrix} 1+\mathfrak p^r &\mathfrak p^{r} \\ \mathfrak p^{r}& 1+\mathfrak p^r\end{smallmatrix}\right]$. Clearly $U^n$ lies in this subgroup for sufficiently large $n$, as does $\ol{U}^{t^{-n}}$. So, for sufficiently large $n$, $\pi$ is trivial on $U^n$, while $\pi^{t^n}$ is trivial on $\ol{U}$. We thus have for large $n$ that 
\[\Hom_L(\pi^{\ol U},(\pi^{t^n})^{U^n})\cong
\Hom_{L\ol{U}}(\pi, (\pi^{t^n})^{U^n})\cong\Hom_{U^n L\ol U}(\pi,\pi^{t^n}).\]
The second assertion follows immediately from the first and Lemma \ref{r_Schur_lemma}.
\end{proof}

\begin{proposition}\label{r_Hecke_proposition} Let $\pi$ be an irreducible representation of $J$. Then
\[ \pres\pi=0 \iff \dim\left(\End_G(\ind_J^G\pi)\right)<\infty.\]
\end{proposition}

\begin{proof} The Mackey formula gives
\[ \End_G(\ind_J^G\pi)\cong \bigoplus_{g\in J\backslash G/J} \Hom_{J\cap J^g}(\pi,\pi^g).\]
Let $w=\left[\begin{smallmatrix} 0 & -1\\ 1& 0\end{smallmatrix}\right]$. According to the Bruhat decomposition, $\{ t^n,\ t^nw\ |\ n\in\Z\}$ is a set of representatives for the double-coset space $J\backslash G /J$  \cite[II.1.7]{Serre_Trees}. 

If $\pres\pi\neq 0$, then  Lemma \ref{r_asymptotic_lemma} ensures that the space $\Hom_{J\cap J^{t^n}}(\pi,\pi^{t^n})$ is nonzero for all $n>>0$. Thus $\End_G(\ind_J^G\pi)$ is infinite-dimensional in this case.

For the converse, suppose that $\pres \pi=0$. Lemma \ref{r_asymptotic_lemma} implies that the cosets $Jt^n J$, for $n\geq 0$, contribute only finitely many dimensions to $\End_G(\ind_J^G\pi)$. Since $\Hom_{J\cap J^{t^n}}(\pi,\pi^{t^n})\cong \Hom_{J^{t^{-n}}\cap J}(\pi^{t^{-n}},\pi)$, the same is true for $n\leq 0$. A small modification of Lemma \ref{r_asymptotic_lemma} shows that the contribution of the double cosets $Jt^nwJ$ is likewise finite-dimensional.
\end{proof}

\begin{remark} The vanishing of $\pres\pi$ does not guarantee that $\ind_J^G\pi$ is a supercuspidal representation of $G$. For example, consider the groups 
$B(\mathfrak f)=\left[\begin{smallmatrix} \mathfrak f^\times &\mathfrak f{\phantom{^\times}}\\ 0{\phantom{^\times}}&\mathfrak f^\times\end{smallmatrix}\right]$ and $N(\mathfrak f)=\left[\begin{smallmatrix} 1&\mathfrak f\\ 0 & 1\end{smallmatrix}\right]$, and let
$\psi$ be a nontrivial one-dimensional representation of $N(\mathfrak f)$. Since $B(\mathfrak f)$ is a quotient of $J$, the representation $\pi=\ind_{N(\mathfrak f)}^{B(\mathfrak f)}\psi$  may be inflated to a representation of $J$. A Mackey-formula computation shows that $\pi^U=\pi^{N(\mathfrak f)}=0$, and so $\pres\pi=0$.
Now, $\pi$ does contain a nonzero vector fixed by the diagonal subgroup $M(\mathfrak f)$: namely, the function 
$f\left(\left[\begin{smallmatrix} x& y\\ 0 & x^{-1}\end{smallmatrix}\right] \right)= \psi(xy)$. The quotient map $J\to B(\mathfrak f)$ sends $J\cap J^w$ onto $M(\mathfrak f)$, and so we have $\pi^{J\cap J^w}\neq 0$. An application of the Mackey formula then gives $(\ind_J^G\pi)^J\neq 0$, and so $\ind_J^G\pi$ has a nonzero summand in the unramified principal series \cite[Lemma 4.7]{Borel_Iwahori}.

It is true, on the other hand, that the cuspidality of $\ind_J^G\pi$ implies $\pres\pi=0$: indeed, if $\pres\pi=\rho\neq 0$, then the pair $(J,\pi)$ is a type for the (non-cuspidal) Bernstein component $[M,\rho]_G$ of $G$ \cite{Kutzko_SL2}.
\end{remark}

\subsubsection*{Parahoric Induction}\label{parahoric_induction_section}

We continue to consider the Iwahori subgroup $J\subset \SL_2(F)$, with its decomposition $J=UL\ol{U}$. Let $K=\SL_2(\mathcal O)$, and define a functor \[\pind_{U,\ol{U}}^K:\Mod (L)\to\Mod (K),\qquad \pind_{U,\ol{U}}^K = \ind_J^K\pind_{U,\ol{U}}.\]
This is an example of \emph{parahoric induction}; see \cite{Dat_Finitude} for the general definition.

The family of representations $\pind_{U,\ol{U}}^K\rho$, as $\rho$ ranges over the irreducibles of $L$, may be considered a kind of principal series for $K$. We will show that the irreducibility and intertwining properties of these representations are exactly analogous to those of the principal series for $\SL_2(F)$ (as explained in \cite{GGP-S}, for instance).

\begin{lemma}\label{Mackey_lemma} Suppose that $I$ and $I'$ are closed subgroups of a compact totally disconnected group, having Iwahori decompositions $I=WM\ol{W}$ and $I'=VM\ol{V}$, where $V\subseteq W$ and $\ol{W}\subseteq \ol{V}$. Then $\Hom_{I\cap I'}(\pind_{W,\ol{W}}\rho,\pind_{V,\ol{V}}\tau)\cong \Hom_M(\rho,\tau)$ for all $\rho,\tau\in\Mod(M)$.
\end{lemma} 

\begin{proof} Let $H=I\cap I'$. This group has an Iwahori decomposition $H=YM\ol{Y}$, where $Y=V$ and $\ol{Y}=\ol{W}$. (We write $Y$ and $\ol{Y}$ in an attempt to avoid ambiguity in the notation; so, for example, $\pind_Y$ is a functor from $\Mod(M)$ to $\Mod(H)$, while $\pind_V$ is a functor from $\Mod(M)$ to $\Mod(I')$.)

Restriction of functions from $I$ to $H$ gives an $H$-equivariant isomorphism ${\pind_{W}} {\rho}\xrightarrow{\cong}\pind_{Y}\rho$; similarly $\pind_{\ol{V}}\tau\xrightarrow{\cong} \pind_{\ol{Y}}\tau$. Embedding $\pind_{W,\ol{W}}\rho\subseteq \pind_{W}\rho$ and $\pind_{V,\ol{V}}\tau\subseteq \pind_{\ol{V}}\tau$, we obtain an injective map
\begin{equation}\label{Mackey_proof_eq1} \Hom_H(\pind_{W,\ol{W}}\rho,\pind_{V,\ol{V}}\tau)\into \Hom_H(\pind_Y\rho,\pind_{\ol{Y}}\tau)\xrightarrow{\cong} \Hom_M(\rho,\tau);\end{equation}
the last isomorphism holds by Frobenius reciprocity, as in Lemma \ref{Centre_lemma}.

On the other hand, restriction of functions from $I$ to $H$ gives a surjective, $H$-equivariant map $\pind_{\ol{W}}\rho\to \pind_{\ol{Y}}\rho$ making the diagram
\[
\xymatrix{
\pind_W\rho \ar[r]^-{I_{\ol{W}}} \ar[d]_-{\text{restrict}} & \pind_{\ol{W}}\rho \ar[d]^-{\text{restrict}}\\
\pind_Y\rho \ar[r]^-{I_{\ol{Y}}} & \pind_{\ol{Y}}\rho
}
\]
commute. This diagram exhibits $\pind_{Y,\ol{Y}}\rho$ as a quotient of $\pind_{W,\ol{W}}\rho$; a similar argument shows that $\pind_{Y,\ol{Y}}\tau$ is a quotient of $\pind_{V,\ol{V}}\tau$. We therefore have an injective map
\begin{equation}\label{Mackey_proof_eq2} \Hom_M(\rho,\tau)\xrightarrow{\cong} \Hom_H(\pind_{Y,\ol{Y}}\rho,\pind_{Y,\ol{Y}}\tau)\into \Hom_H(\pind_{W,\ol{W}}\rho,\pind_{V,\ol{V}}\tau);\end{equation}
the first isomorphism holds by Proposition \ref{parahoric_basic_proposition}.

Since $\Hom_M(\rho,\tau)$ is finite-dimensional when $\rho$ and $\tau$ are, the injective maps \eqref{Mackey_proof_eq1} and \eqref{Mackey_proof_eq2} are isomorphisms.
\end{proof}

Applied to the Iwahori subgroup in $\SL_2(F)$, Lemma \ref{Mackey_lemma} gives the following Mackey-type formula for parahoric induction and restriction. (The corresponding formula in the general case is the subject of ongoing work with Ehud Meir and Uri Onn.)

\begin{lemma}\label{SL2_Mackey_lemma} Let $\rho$ and $\tau$ be representations of $L$. Then 
\[ \Hom_K(\pind_{U,\ol{U}}^K\rho,\pind_{U,\ol{U}}^K\tau)\cong \Hom_L(\rho,\tau)\oplus \Hom_L(\rho,\tau^w).\]
\end{lemma}

\begin{proof} Using the Mackey formula and the Bruhat decomposition $K=J\sqcup JwJ$, we find
\[ \Hom_K(\pind_{U,\ol{U}}^K\rho,\pind_{U,\ol{U}}^K\tau)\cong \Hom_J(\pind_{U,\ol{U}}\rho,\pind_{U,\ol{U}}\tau)\oplus \Hom_{J\cap J^w} (\pind_{U,\ol{U}}\rho, (\pind_{U,\ol{U}}\tau)^w).\]
The first summand is isomorphic to $\Hom_L(\rho,\tau)$, by Proposition \ref{parahoric_basic_proposition}. We have $(\pind_{U,\ol{U}}\tau)^w\cong \pind_{U^w,\ol{U}^w}(\tau^w)$, and so Lemma \ref{Mackey_lemma} implies that the second summand is isomorphic to $\Hom_L(\rho,\tau^w)$.
\end{proof}

An application of Schur's lemma then gives:

\begin{proposition}\label{principal_proposition} Let $\rho$ and $\rho'$ be irreducible representations of $L$.
\begin{enumerate}[label={\rm(\arabic*)}, leftmargin=0cm,itemindent=2em,labelwidth=\itemindent,labelsep=0cm,align=left]
\item $\pind_{U,\ol{U}}^K\rho$ is irreducible if and only if $\rho\not\cong \rho^w$.
\item If $\rho\cong\rho^w$, then $\pind_{U,\ol{U}}^K\rho$ is a sum of two inequivalent irreducibles.
\item $\pind_{U,\ol{U}}^K\rho \cong\pind_{U,\ol{U}}^K\rho^w$.
\item $\Hom_K(\pind_{U,\ol{U}}^K\rho, \pind_{U,\ol{U}}^K\rho')=0$\ \ if\ \ $\rho'\not\cong\rho$ or $\rho^w$. \hfill\qed
\end{enumerate}
\end{proposition}

\section{Parahoric Induction and Chamber Homology for $\SL_2(F)$}\label{chamber_section}

\subsubsection*{Background}

Keep the notation $G$, $J$, $K$, $L$, $U$, $\ol{U}$, $F$, etc., from the previous section. We also set 
\[M=\begin{bmatrix} F & 0\\ 0 & F\end{bmatrix},\ N=\begin{bmatrix} 1& F\\ 0 & 1\end{bmatrix},\ \overline{N}= \begin{bmatrix} 1& 0\\ F&1\end{bmatrix},\ P=\begin{bmatrix}F & F\\ 0 & F\end{bmatrix},\ K'=\begin{bmatrix} \mathcal O & \mathfrak p^{-1}\\ \mathfrak p & \mathcal O\end{bmatrix},\]
\[V=\begin{bmatrix} 1 & \mathfrak p^{-1}\\ 0 & 1\end{bmatrix},\  w=\begin{bmatrix} 0&-1\\ 1&0\end{bmatrix},\ t=\begin{bmatrix} \varpi^{-1} & 0 \\ 0 &\varpi\end{bmatrix},\ \text{and } W=\{1,w\}/\pm 1\]
($\mathfrak p^{-1}$ means $\varpi^{-1}\mathcal O$). We consider the normalised Jacquet functors $\pind_M^G$ and $\pres^G_M$ of parabolic induction and Jacquet restriction along $P$ \cite[VI.1]{Renard}. 

Work of Bernstein \cite{Bernstein-Deligne} and Keller \cite{Keller} implies that the Hochschild homology groups $\HH_*(\H(G))$ and $\HH_*(\H(M))$ may be defined in terms of the categories of finitely generated modules over $\H(G)$ and $\H(M)$, respectively: see \cite{Crisp_compact}. The Jacquet functors preserve the subcategories of finitely generated modules in $\Mod(G)$ and $\Mod(M)$, and so they induce natural maps between $\HH_*(\H(G))$ and $\HH_*(\H(M))$.

We let $G_c$ denote the union of the compact subgroups of $G$. This set is open, closed, and conjugation-invariant in $G$, and so it determines a direct-summand $\HH_*(\H(G))_c$ of $\HH_*(\H(G))$: see \cite{Blanc-Brylinski}. The map $\pres^G_M$ sends $\HH_*(\H(G))_c$ to $\HH_*(\H(M))_c$, and the map $\pind_M^G$ sends $\HH_*(\H(M))_c$ to $\HH_*(\H(G))_c$ \cite[Corollaries 3.12 and 3.19]{Crisp_compact}. 

The Bruhat-Tits building $X$ of $G$ is an infinite, locally finite, connected tree, on which $G$ acts properly, simplicially, and without inversions \cite[II.1]{Serre_Trees}. The action is transitive on the set $X^1$ of edges, and has two orbits in the vertex-set $X^0$. The Iwahori subgroup $J$ is the isotropy group of an edge, whose vertices have isotropy groups $K$ and $K'$. The chamber homology $\h_*^G(X)$ of $X$ is, by definition, the homology of the following chain complex \cite{Baum-Higson-Plymen_SL2}:
\begin{equation}\tag*{$C_*^G(X)$:}
R(J)\xrightarrow{\partial_G} R(K)\oplus R(K')\qquad \partial_G(\pi)=\ind_J^K\pi\oplus -\ind_J^{K'}\pi
\end{equation}

The building $Y$ of $M$ identifies with an apartment (i.e., a line) in $X$. With respect to the decomposition $M\cong L\times \langle t\rangle$, $L$ acts trivially on $Y$ while $t$ translates the $i$th vertex to the $(i+2)$nd. The chamber homology $\h_*^M(Y)$ of $Y$ is the homology of the chain complex
\begin{equation}
\tag*{$C_*^M(Y)$:} R(L)\oplus R(L) \xrightarrow{\partial_M} R(L)\oplus R(L)\quad \partial_M(\rho_0,\rho_1)=(\rho_0+\rho_1,-\rho_0-\rho_1)
\end{equation}
In pictures, showing the (oriented) quotient complexes $Y/M$ and $X/G$ labelled by their respective coefficient systems:
\begin{flalign*}& \xymatrix@C=40pt@R=80pt{ 
Y/M: & & \bullet \ar@{}[]+L(4)*{R(L)} \ar@/^2pc/[0,2]^*-{R(L)} \ar@/_2pc/[0,2]_*-{R(L)}&& \circ \ar@{}[]+R(4)*{R(L)}  \\
X/G: & & \bullet \ar@{}[]+L(4)*{R(K')}  \ar[0,2]^-*{R(J)} && 
\circ \ar@{}[]+R(4)*{R(K)} 
}&
\end{flalign*}

There are canonical isomorphisms $\h_*^G(X)\cong \HH_*(\H(G))_c$ and $\h_*^M(Y)\cong \HH_*(\H(M))_c$: see \cite{Higson-Nistor} and \cite{Schneider}. (Part of the argument is also outlined below.) The action of the Weyl group $W$ on $Y$ and $M$ induces an action on $C_*^M(Y)$ as follows: in degree zero, $w(\rho_0,\rho_1)=(\rho_0^w,\rho_1^w)$. In degree one, $w(\rho_0,\rho_1)=(\rho_1^w,\rho_0^w)$. The induced action on chamber homology agrees, under the embedding $\h_*^M(Y)\into \HH_*(\H(M))$, with the one given by the action of $W$ on $M$ by conjugation.

\subsubsection*{Jacquet functors in chamber homology}

\begin{definition} Let $\pind_c:\h_*^M(Y)\to \h_*^G(X)$ and $\pres_c:\h_*^G(X)\to \h_*^M(Y)$ be the maps induced by restricting the Jacquet functors $\pind_M^G$ and $\pres^G_M$ to the compact part of Hochschild homology. Thus $\pind_c$ and $\pres_c$ are the unique maps making the diagrams
\[
\xymatrix{ \h_*^M(Y) \ar[r]^-{\pind_c} \ar[d]_-{\cong} & \h_*^G(X) \ar[d]^-{\cong}\\
\HH_*(\H(M))_c \ar[r]^-{\pind_M^G} & \HH_*(\H(G))_c } \quad\text{and}\quad 
\xymatrix{ \h_*^G(X) \ar[r]^-{\pres_c} \ar[d]_-{\cong} & \h_*^M(Y) \ar[d]^-{\cong}\\
\HH_*(\H(G))_c \ar[r]^-{\pres^G_M} & \HH_*(\H(M))_c }
\]
commute. 
\end{definition}

Recall that we have defined $\pind_{U,\ol{U}}^K:\Mod(L)\to \Mod(K)$ as the composition $\ind_J^K\pind_{U,\ol{U}}$. We likewise define $\pind_{U,\ol{U}}^{K'}\coloneq \ind_J^{K'}\pind_{U,\ol{U}}$. 

\begin{definition-lemma} \label{RI_definition}
The following diagrams commute, and therefore define maps of complexes $\Pind:C_*^M(Y)\to C_*^G(X)$ and $\Pres:C_*^G(X)\to C_*^M(Y)$ (in that order).
\[ \xymatrix@C=5pt@R=40pt{ (\rho_0,\rho_1) \ar@{|-{>}}[d] & R(L)\oplus R(L) \ar[0,5]^-{\partial_M} \ar[d] &&&&&  R(L)\oplus R(L) \ar[d] &(\rho_0,\rho_1) \ar@{|-{>}}[d] \\
\pind_{U,\ol{U}}\rho_0 + \pind_{U,\ol{U}}\rho_1^w & R(J) \ar[0,5]^-{\partial_G} &&&&& R(K)\oplus R(K') & (\pind_{U,\ol{U}}^K\rho_0,\pind_{U,\ol{U}}^{K'}\rho_1)} \]
\[ \xymatrix@C=5pt@R=40pt{ \pi \ar@{|-{>}}[d] & R(J) \ar[0,5]^-{\partial_G} \ar[d] &&&&&  R(K)\oplus R(K') \ar[d] &(\pi_0,\pi_1) \ar@{|-{>}}[d] \\
\left(\pi^U, (\pi^{\ol{U}})^w\right) & R(L)\oplus R(L) \ar[0,5]^-{\partial_M} &&&&& R(L)\oplus R(L) & (\pi_0^{U}, \pi_1^{V})} \]
\end{definition-lemma} 

\begin{proof}
The first diagram commutes by virtue of the equality $\pind_{U,\ol{U}}^K\rho\cong\pind_{U,\ol{U}}^K\rho^w$ from Proposition \ref{principal_proposition}, along with the analogous equality for $K'$.

In the second diagram we are asserting that for each representation $\pi$ of $J$,
\begin{equation}\label{RI_definition_equation} \left(\ind_J^K\pi\right)^U \cong \pi^{U}\oplus \left(\pi^{\ol{U}}\right)^w
\end{equation}
and similarly for induction to $K'$.
An application of the Mackey formula gives
\[ \left(\ind_J^K\pi\right)^{U}\cong \pi^{U}\oplus \left(\ind_{J\cap J^w}^J \pi^w\right)^{U},\] and a character computation confirms that the second summand is isomorphic to ${(\pi^{\ol{U}})}^w$.
\end{proof}

\begin{theorem}\label{parahoric_chamber_theorem} $\Pind=\pind_c$ and $\Pres=\pres_c$ as maps on chamber homology.
\end{theorem}

The proof of Theorem \ref{parahoric_chamber_theorem} occupies most of the remainder of the paper. 

\subsubsection*{Proof that $\Pres=\pres_c$}

An explicit formula for the map $\pres^G_M$ on Hochschild homology is given, for a general reductive group $G$ and Levi subgroup $M$, in \cite{Crisp_compact}. The same map appeared earlier in \cite{Nistor}, where Nistor computes the corresponding map on smooth group homology. Let us recall these results, in summary.

Let $\H(G_c)$ denote the space of locally constant, compactly supported functions on $G_c$, considered as a $G$-module under the adjoint action. As observed in \cite{Higson-Nistor} and \cite{Schneider}, $C_*^G(X)$ is isomorphic to the $G$-coinvariants of the following projective resolution of $\H(G_c)$:
\begin{equation}
\tag*{$C_*(X,G)$:} \bigoplus_{e\in X^1} \H(G_e) \xrightarrow{\partial} \bigoplus_{v\in X^0}\H(G_v)
\end{equation}
(The boundary $\partial$ and the augmentation $C_0(X,G)\to \H(G_c)$ are given by extending functions by zero.)
It follows that $\h_*^G(X)\cong \h_*(G,\H(G_c))$, the right-hand side being smooth group homology (the left-derived functor of $G$-coinvariants on $\Mod(G)$). Blanc and Brylinski show in \cite{Blanc-Brylinski} that there is a canonical isomorphism $\h_*(G,\H(G_c))\cong \HH_*(\H(G))_c$, whence the identification of chamber homology with the compact part of Hochschild homology. Similar considerations apply to $M$: $C_*^M(Y)$ is the complex of $M$-coinvariants of the complex $C_*(Y,M)$ of simplicial chains on $Y$ with coefficients in $\H(L)$, giving $\h_*^M(Y)\cong \h_*(M,\H(L))$ (note that $L=M_c$).

Let $\delta$ be the modular function on $P$, characterised by $\dd (pq) =\delta(q) \dd p$ for any left Haar measure $\dd p$ on $P$. For each $\rho\in\Mod(M)$, $\rho_{\delta^{1/2}}\coloneq \rho\otimes_{\C}\delta^{1/2}$ denotes the twisting of $\rho$ by the one-dimensional representation $\delta^{1/2}$. For each representation $\pi\in\Mod(G)$, the idempotent $\pi(e_K):\pi\to\pi$ descends to a well-defined map $\pi_G\to (\pres^G_M(\pi)_{\delta^{1/2}})_M$ 
between the $G$-coinvariants of $\pi$ and the $M$-coinvariants of $\pres^G_M(\pi)_{\delta^{1/2}}$. (Here one appeals to the Iwasawa decomposition $G=KMN$.) This map is natural in $\pi$, and so it lifts to a natural transformation of derived functors,
\[ \kappa:\h_*(G,\pi)\to \h_*(M,\pres^G_M(\pi)_{\delta^{1/2}}).\]
The ``Harish-Chandra transform''
\[\Psi:\H(G_c)\to \H(L),\qquad \Psi(f)(l)=\int_N f(nl)\, \dd n\]
descends to an $\Ad_M$-equivariant map $\pres^G_M\H(G_c)_{\delta^{1/2}}\to \H(L)$. The Jacquet restriction $\pres_c:\h_*^G(X)\to \h_*^M(Y)$ is then equal to the composition
\[ \h_*^G(X)\xrightarrow{\cong} \h_*(G,\H(G_c))\xrightarrow{\Psi\circ\kappa} \h_*(M,\H(L))\xrightarrow{\cong} \h_*^M(Y).\]
See \cite{Nistor} and \cite{Crisp_compact} for details.

\begin{proof}[Proof that $\Pres=\pres_c$ in Theorem \ref{parahoric_chamber_theorem}]
The inclusion of $Y$ into $X$ gives an isomorphism $Y\cong X/N$. It follows that the image of the resolution $C_*(X,G)$ under the functor $\pres^G_M(\argument)_{\delta^{1/2}}$ is isomorphic to
\begin{equation}\tag*{$C_*(Y,\pres G)$:}
\bigoplus_{e\in Y^1} \H(G_e)_{N_e} \to \bigoplus_{v\in Y^0} \H(G_v)_{N_v},
\end{equation}
the subscripts $N_e$ and $N_v$ denoting coinvariants with respect to the adjoint action. The maps
\[ \Psi_s:\H(G_s)_{N_s}\to \H(L),\qquad \Psi_s(f)(l)=\int_{N_s} f(nl)\, \dd n,\]
where $s$ ranges over the simplices in $Y$, provide a lift of $\Psi$ to a map of resolutions, $C_*(Y,\pres G)\to C_*(Y,M)$. 
We claim that the composition
\begin{equation}\label{r_composition_equation} C_*^G(X)\xrightarrow{\cong} C_*(X,G)_G \xrightarrow{\kappa} C_*(Y,\pres G)_M \xrightarrow{\Psi} C_*(Y,M)_M \xrightarrow{\cong} C_*^M(Y)
\end{equation}
is equal to $\Pres$.

For example, let $\pi$ be a representation of $K$, viewed as a chain in $C_0^G(X)$. The corresponding chain in $C_0(X,G)$ is the function $\ch_\pi\in \H(K)$; recall that $K$ is the isotropy group of a vertex in $X$. This vertex lies in $Y$, and so the map $\kappa$ simply acts on $\ch_\pi$ by averaging over the adjoint action of $K$; $\ch_\pi$ is already $\Ad_K$-invariant, so $\kappa(\ch_\pi)=\ch_\pi\in \H(K)_U$. The map $\Psi:\H(K)_U\to \H(L)$ sends $\ch_\pi$ to $\ch_{\pi^U}$, and so \eqref{r_composition_equation} equals $\Pres$ as maps $R(K)\to R(L)$. The computations for $R(K')$ and $R(J)$ are only slightly more involved (because the cycles in question are not a priori $K$-invariant). We shall not present the details here.  
\end{proof}

\subsubsection*{Proof that $\Pind=\pind_c$}

Unlike the preceding section, whose methods apply to general reductive $G$ and Levi subgroup $M$, our proof that $\Pind=\pind_c$ relies on some special features of $\SL_2$: $\h_*^G(X)$ is nonzero only in degrees zero and one, and $\pres_c:\h_1^G(X)\to \h_1^M(Y)$ is an isomorphism onto the space of $W$-invariants in $\h_1^M(Y)$; see \cite{Nistor} and \cite{Crisp_compact}.

\begin{theorem} {\rm \cite[Theorem 5.2]{Bernstein-Zelevinsky_InducedI}} $\pres_c\pind_c=1+w$ as endomorphisms of $\h_*^M(Y)$.
\end{theorem}

\begin{proof} The cited result of Bernstein and Zelevinsky implies that the functor $\pres^G_M\pind_M^G$ on $\Mod(M)$ has a natural filtration with quotients $1$ and $w$. This filtration becomes a sum in Hochschild homology \cite{Crisp_compact}.
\end{proof} 

\begin{proposition}\label{RI_proposition} $\Pres\Pind=1+w$ as endomorphisms of $C_*^M(Y)$.
\end{proposition}

\begin{proof}
For each irreducible $\rho$ of $L$, one has 
\[ \left(\pind_{U,\ol{U}}^K \rho\right)^U\cong (\pind_{U,\ol{U}}\rho)^U \oplus \left((\pind_{U,\ol{U}}\rho)^{\ol{U}}\right)^w \cong \rho\oplus\rho^w;\]
the first isomorphism is \eqref{RI_definition_equation}, the second follows from Lemma \ref{r_Schur_lemma}. This (and the corresponding computation for $K'$) shows that $\Pres\Pind=1+w$ in degree zero. In degree one, Lemma \ref{r_Schur_lemma} gives $\Pres\Pind=1+w$ immediately.
\end{proof}

\begin{proof}[Proof that $\Pind=\pind_c$ in Theorem \ref{parahoric_chamber_theorem}]
We have shown that $\pres_c \Pind=\Pres\Pind=1+w=\pres_c\pind_c$. Since $\pres_c$ is one-to-one in degree one, this gives $\Pind=\pind_c$ as maps $\h_1^M(Y)\to \h_1^G(X)$.

The equality in degree zero is deduced from a theorem of Dat, as follows. $\HH_0(\H(G))$ is a quotient of the complex vector space $\mathcal V_G$ with basis consisting of pairs $[\sigma,T]$, where $\sigma$ is a finitely generated projective $G$-module, and $T\in \End_G(\sigma)$  (\cite[1.3]{Dat_K0}, \cite[Proposition 2.7]{Crisp_compact}). The inclusion $\h_0^G(X)\into \HH_0(\H(G))$ is then the one induced in homology by
\[ R(K)\to \mathcal V_G,\qquad \pi\mapsto [\ind_K^G\pi,\id]\]
and by the corresponding map $R(K')\to \mathcal V_G$. Similar considerations apply to $M$, and the map $\pind^G_M:\HH_0(\H(M))\to \HH_0(\H(G))$ is the one induced by 
\[ \mathcal V_M\to \mathcal V_G\qquad [\sigma,T]\mapsto[\pind^G_M\sigma,\pind^G_M T].\]
So the theorem in degree zero follows from the assertion that
\[ \pind_M^G \ind_L^M\rho \cong \ind_K^G \pind_{U,\ol{U}}^K\rho,\] naturally with respect to $\rho\in \Mod(L)$, and similarly for $K'$. This assertion is a special case of \cite[(1.4)]{Dat_Finitude}.
\end{proof}

The following description of $\h_1^G(X)$ follows immediately from Theorem \ref{parahoric_chamber_theorem}. Together with Proposition \ref{SL2_iz_proposition}(1), this gives a new proof of \cite[Proposition 9.3]{Baum-Higson-Plymen_SL2}, and also explains the resemblance with principal-series characters observed in \cite[p.17]{Baum-Higson-Plymen_SL2}.

\begin{corollary} $\h_1^G(X)$ has a basis consisting of cycles $\pind_{U,\ol{U}}(\rho)-\pind_{U,\ol{U}}(\rho^w)\in R(J)$, where $\rho$ ranges over a set of representatives for the two-element orbits of $W$ on the set of irreducible representations of $L$.
\end{corollary}

\begin{proof} The map $\pres_c:\h_1^G(X)\to \h_1^M(Y)$ is injective, with range equal to the space of $W$-invariants in $\h_1^M(Y)$. The cycles $c_\rho\coloneq (\rho-\rho^w,\rho^w-\rho)\in C_1^M(Y)$, for $\rho$ as in the statement of the corollary, constitute a basis for the latter space, and Proposition \ref{RI_proposition} shows that
\[ R(\pind_{U,\ol{U}}(\rho)-\pind_{U,\ol{U}}(\rho^w))=\frac{1}{2}RI(c_\rho)=c_\rho.\qedhere\]
\end{proof}

\subsubsection*{The case of $\SL_n$}

The definitions of $\Pres$ and $\Pind$ make sense also for $G=\SL_n(F)$, $M$ the diagonal subgroup. For example, in degree $n-1$ one sets
\[ \Pind:R(L)^n\to R(J),\qquad (\rho_0,\ldots,\rho_{n-1})\mapsto \sum_{w_i\in W} \pind_{U,\ol{U}}\rho_i^{w_i},\] where $J=UL\ol{U}\subset G$ is the standard Iwahori subgroup, and $W=N_G(M)/M$ is the Weyl group, which acts simply transitively on the set of chambers in a fundamental domain for the action of $M$ on its apartment. In degree zero, 
\[ \Pind:R(L)^n\to \bigoplus_{i=0}^{n-1}R(K_i),\qquad (\rho_0,\ldots,\rho_{n-1})\mapsto (\pind_{U,\ol{U}}^{K_0}\rho_0,\ldots,\pind_{U,\ol{U}}^{K_{n-1}}\rho_{n-1}),\] where $K_0,\ldots,K_{n-1}$ are the isotropy groups of the vertices of the chamber stabilised by $J$, and $\pind_{U,\ol{U}}^{K_i}=\ind_J^{K_i}\pind_{U,\ol{U}}$. The above proof carries over to give the following partial result:

\begin{proposition} Let $G=\SL_n(F)$, and let $M\subset G$ be the diagonal subgroup. Define maps $\Pres:C_*^G(X)\to C_*^M(Y)$ and $\Pind:C_*^M(Y)\to C_*^G(X)$ as above. Then $\Pres=\pres_c$ as maps $\h_*^G(X)\to \h_*^M(Y)$, and $\Pind=\pind_c$ as maps  $\h_0^M(Y)\to \h_0^G(X)$ and $\h_{n-1}^M(Y)\to \h_{n-1}^G(X)$.\hfill\qed
\end{proposition}

Replacing the diagonal subgroup by a larger Levi subgroup, for example the $(2\times 1)$-block-diagonal subgroup of $\SL_3(F)$,
one can still use parahoric induction to define a candidate for the map $\Pind$. It follows from our joint work (in progress) with Ehud Meir and Uri Onn that this map will no longer commute with the boundary maps; the issue is closely related to Dat's question \cite[Question 2.14]{Dat_Finitude}. It is likely that new tools will be needed in this situation.

\bibliography{parahoric_SL2}{}
\bibliographystyle{plain}

\end{document}